\documentclass[12pt,a4paper]{paper} 

\usepackage{latexsym,amsmath,amsfonts,amscd,amssymb}
\usepackage{lscape}
\usepackage{mathdots}
\usepackage{array}
\usepackage{pdfsync}
\usepackage[all]{xy}
\usepackage{amsthm}
\usepackage{multirow}
\usepackage[T1]{fontenc} 

\newcommand{\st}{\;|\;}

\newcommand{\ts}{\times}
\newcommand{\op}{\oplus}
\newcommand{\ot}{\otimes}

\newcommand{\Z}{{\mathbb{Z}}}
\newcommand{\Q}{{\mathbb{Q}}}
\newcommand{\R}{{\mathbb{R}}}
\newcommand{\C}{{\mathbb{C}}}

\newcommand{\cC}{\mathcal{C}}

\newcommand{\cL}{\mathcal{L}}

\newcommand{\fso}{\mathfrak{so}}
\newcommand{\fgdiff}{\mathfrak{gdiff}}

\newcommand{\al}{\alpha}
\newcommand{\be}{\beta}

\newcommand{\lam}{\lambda}

\renewcommand{\phi}{\varphi}

\newcommand{\la}{\langle}
\newcommand{\ra}{\rangle}

\newcommand{\Diff}{\mathrm{Diff}}
\newcommand{\GDiff}{\mathrm{GDiff}}

\newcommand{\SU}{\mathrm{SU}}

\newcommand{\SO}{\mathrm{SO}}

\newcommand{\SSS}{\mathrm{S}}
\newcommand{\Spin}{\mathrm{Spin}}

\DeclareMathOperator{\rk}{rk}

\DeclareMathOperator{\id}{id}

\DeclareMathOperator{\End}{End}

\DeclareMathOperator{\Aut}{Aut}

\DeclareMathOperator{\Ann}{Ann}

\theoremstyle{theorem}
\newtheorem{theorem}{Theorem}[section]

\newtheorem{lemma}[theorem]{Lemma}
\newtheorem{proposition}[theorem]{Proposition}

\theoremstyle{definition}
\newtheorem{definition}[theorem]{Definition}
\theoremstyle{remark}

\newtheorem{remark}[theorem]{Remark}

\title{$B_n$-generalized geometry\\ and $G_2^2$-structures}   

\author{Roberto Rubio \thanks{Mathematical Institute, 24-29 St Giles, Oxford OX1 3LB, UK, rubio@maths.ox.ac.uk}}

\date{today}

\begin{document}

\maketitle

\begin{abstract}
  We introduce the concept of $G_2^2$-structure on an orientable
  $3$-manifold $M$ using the setting of generalized geometry of type
  $B_n$, study their local deformation by making use of a Moser-type
  argument, and give a description of the cone of $G_2^2$-structures.
\end{abstract}


\section{Introduction}
\label{sec:introduction}

Generalized geometry was originally introduced in
\cite{hitchin-CY:2003} as, naively, the differential geometry
resulting from replacing the tangent bundle $T$ of a manifold $M$ with
the sum of the tangent and cotangent bundles, $T\op T^*$, which is
naturally endowed with an $\SO(n,n)$-structure. Classical concepts have
then generalized analogues, such as generalized metrics and
generalized Calabi-Yau or generalized complex structures. An
interesting feature of this geometry is that the bundle of
differential forms $\bigwedge^\bullet T^* M$ becomes a bundle of
spinors, in which some of these structures are formulated. For
instance, a generalized Calabi-Yau structure is given by a closed
section of $\bigwedge^{ev} T^* M\otimes \C$ or $\bigwedge^{od} T^*
M\otimes \C$ consisting of pure spinors.

The generalized tangent space $T\op T^*$ can be further modified by
adding new pieces. The simplest one is the rank $1$ trivial bundle
over $M$, which we denote by $1$. Since the natural metric of $T\op
T^*\op 1$ has signature $(n+1,n)$, the group of symmetries becomes
$\SO(n+1,n)$. As this group is of Lie type $B_n$, we call this
geometry generalized geometry of type $B_n$, from now on
$B_n$-geometry. Correspondingly, ordinary generalized geometry is
called $D_n$-geometry. Exceptional geometries based on the split real
forms $E^n_n$ have also been studied as, for example, in
\cite{hull:2007}.

$B_n$-geometry was originally introduced by Baraglia in
\cite{baraglia:2011} (Section 2.4). It also arises as a particular
case of the more general situation studied in
\cite{chen-stienon-xu:2009}. Section \ref{sec:Bn-gengeo} of the
present work is devoted to stating the basic features of
$B_n$-geometry.

In Section \ref{sec:G22-str}, we introduce $G_2^2$-structures on an
orientable $3$-manifold $M$ as suggested by
Baraglia. $G_2^2$-structures are defined by analogy with generalized
Calabi-Yau structures. They are given by a closed section $\rho$ of
$\bigwedge^\bullet T^*$ consisting of non-pure spinors. We consider the existence and equivalence of $G_2^2$-structures on compact orientable $3$-manifolds. While
$G_2^2$-structures with non-vanishing degree $0$ component,
$\rho_0\neq 0$, exist on any $3$-manifold, those with $\rho_0 = 0$ only
exist on orientable mapping tori. In fact any mapping torus of an
orientable surface can be endowed with such a $G_2^2$-structure
(Theorem \ref{theo:map-tor-twisted-G22-str}). In Section
\ref{subsec:deformation-G22} we show that the Moser argument in
symplectic geometry can be modified to obtain the result that a small
deformation within the cohomology class does not change the structure
up to generalized diffeomorphism (Theorem
\ref{theo:small-perturbation}). We finish by describing the cone of
$G_2^2$-structures inside $H^\bullet(M)$ in
Theorem~\ref{theo:G22-cone}. \bigskip

The author wishes to thank his supervisor Nigel Hitchin for
introducing him to this subject and for his constant support and
generosity. This work was been possible thanks to a Fellowship for
Graduate Courses in Universities and Colleges funded by Fundaci\'on
Caja Madrid.

\section{$B_n$-generalized geometry}
\label{sec:Bn-gengeo}

\subsection{The Courant algebroid $T\op T^*\op 1$}
\label{subsec:group-gen-diff}

Let $M$ be a differentiable manifold of dimension $n$ with tangent
bundle $T$ and cotangent bundle $T^*$. Let $1$ denote the trivial
bundle of rank $1$ over $M$. Define the $B_n$-generalized tangent
bundle by $T\op T^*\op 1$. The sections of this bundle are called
generalized vector fields and are naturally endowed with a signature
$(n+1,n)$ inner product given~by
$$(X+\xi+\lambda,Y+\eta+\mu)=\frac{1}{2}(i_X\eta + i_Y\xi) + \lambda\mu,$$
where $X+\xi+\lambda, Y+\eta+\mu \in \cC^\infty(T\op T^*\op 1)$.
Together with the canonical orientation on $T\op T^*\op 1$, this
endows $T\op T^*\op 1$ with the structure of an
$\SO(n+1,n)$-bundle. We introduce a Courant bracket on
$\cC^\infty(T\op T^*\op 1)$ via
\begin{equation*}
  \begin{split}
    [X+\xi+\lam,Y+\eta+\mu] = {} & [X,Y]+\cL_X\eta-\cL_Y\xi
    -\frac{1}{2}d(i_X \eta-i_Y \xi)\\ & + \mu d\lam-\lam d\mu +(i_X
    d\mu - i_Y d\lam),
  \end{split}
\end{equation*}
so that $(T\op T^*\op 1, (,), [,])$ is a Courant algebroid in the
sense of \cite{liu-weinstein-xu:1997}.

The infinitesimal orthogonal transformations of $T\op T^*\op 1$ are
given by the elements
$$\left(
\begin{array}{ccc}
  E & \beta & -2\alpha \\
  B & -E^t & -2A \\
  A & \alpha & 0 
\end{array}
\right) \in \cC^\infty(\fso(T\op T^*\op 1))$$ such that
$E\in \End(T)$, $\beta\in \bigwedge^2 T$, $B\in \bigwedge^2 T^*$, the
$B$-field already present in $D_n$-geometry, $\alpha \in T$ and $A\in
T^*$, the $A$-field which will be relevant in $B_n$-geometry. The
exponentiation of a $B+A$-field gives the element
$$(B,A):=\exp(B+A)=\left(
\begin{array}{ccc}
  1 &  &  \\
  B-A\ot A & 1 & -2A \\
  A &  & 1 
\end{array}
\right)\in \cC^\infty(\SO(T\op T^*\op 1)),$$ which acts by
$(B,A)(X+\xi+\lambda) = X+\xi+i_X B-2\lambda A-i_X A\, A+\lambda +i_X
A.$

The composition law of these elements in $\cC^\infty(\SO(T\op T^*\op
1))$ is
$$(B,A)(B',A')= (B+B'+ A\wedge A',A+A').$$
Note that $A$-fields do not commute and their product involves a
$2$-form.

Their action on the Courant bracket is given by the following result.
\begin{proposition}\label{prop:e(B,a)-courant}
  Let $(B,A)\in \cC^\infty(\SO(T\op T^*\op 1))$. For generalized
  vector fields $u=X+\xi+\lam$ and $v=Y+\eta+\mu$, we have
  \begin{equation*}
    \begin{split}
      [(B,A) u ,(B,A) v] = {} & (B,A)[u,v]+i_Yi_X(dB+A\wedge dA) -2i_Y
      i_X dA \cdot A \\ & + i_Y i_X dA + 2(\lambda i_Y dA - \mu i_X
      dA).
    \end{split}
  \end{equation*}
  In particular, the action of $(B,A)$ commutes with the Courant
  bracket if and only if $A$ and $B$ are closed.
\end{proposition}

Define the group
$$\Omega^{2+1}_{cl}(M)=\{(B,A)\in \cC^\infty(\SO(T\op T^*\op 1)) \st B\in \Omega^{2}_{cl}(M),  A\in \Omega^1_{cl}(M)\}.$$
The group $\Omega^2_{cl}(M)$ is a central subgroup in
$\Omega^{2+1}_{cl}(M)$, so $\Omega^{2+1}_{cl}(M)$ is the central
extension $1\to \Omega^{2}_{cl}(M)\to \Omega^{2+1}_{cl}(M) \to
\Omega^{1}_{cl}(M)\to 1.$



\begin{proposition}\label{prop:semidirect-product}
  The group of orthogonal transformations of $T\op T^*\op 1$
  preserving the Courant bracket is $\mathrm{Diff}(M) \ltimes
  \Omega^{2+1}_{cl}(M)=: \mathrm{GDiff}(M)$, called the group of
  generalized diffeomorphisms of $M$. The product is given by
  \begin{align*}
    (f\ltimes (B,A))\circ (g\ltimes (D,C)) & = fg \ltimes
    (g^*B,g^*A)(D,C) \nonumber \\ & = fg \ltimes (g^*B+D+g^*A \wedge
    C,g^*A+C).
  \end{align*}
\end{proposition}

We describe $\fgdiff(M)$, the Lie algebra of $\GDiff(M)$, by
differentiating the action of a smooth one-parameter family of
generalized diffeomorphisms $F_t=f_t\ltimes (B_t,A_t)$ such that
$F_t\circ F_s=F_{t+s}$ and $F_0=\id$. By Proposition
\ref{prop:semidirect-product} and $F_t\circ F_s=F_{t+s}$ we have the
three equations \begin{align*} f_{t+s} &=f_t\circ f_s, & A_{t+s} &=A_s
  + f_s^*A_t, & B_{t+s} &=B_s + f_s^*B_t + f_s^*A_t\wedge A_s.
\end{align*}

The first equation says that $\{f_t\}$ is a one-parameter subgroup of
diffeomorphisms of $M$. Let $X$ be the corresponding vector field. From the second equation, $A_t=\int_0^t f_s^* a\, ds$,
where $a=\frac{dA_t}{dt}_{|t=0}$. And from the third equation,
$$\frac{dB_t}{dt}_{|t=s}=f_s^*\frac{dB_t}{dt}_{|t=0} + f_s^*\frac{dA_t}{dt}_{|t=0}\wedge A_s,$$ so $B_t=\int_0^t (f_s^* b + f_s^* a\wedge A_s) ds$, where $b=\frac{dB_t}{dt}_{|t=0}$ and $A_s$ depends on $a$.

Using the convention $\cL_X Y=-\frac{d}{dt}_{|t=0} f_{t\, *} Y$ for
the Lie derivative of a vector field $Y$, we see that the
infinitesimal action of the one-parameter subgroup $\{F_t\}$ is
\begin{equation*}
  -\frac{d}{dt}_{|t=0} F_{t\, *} (Y+\eta+\mu) = \cL_X (Y+\eta+\mu) - i_Y b + 2\mu a - i_Y a,
\end{equation*}
which only depends on the action of $(X,b,a)$. We thus make the
identification $$\fgdiff(M)=\cC^\infty(T)\op \Omega^2_{cl}(M) \op
\Omega^1_{cl}(M).$$

Conversely, given an infinitesimal generalized diffeomorphism
$(X,b,a)$, we can integrate it to a one-parameter subgroup of
generalized diffeomorphisms using the equations above.

\begin{remark}\label{remark:integrating-time-dependent}
  It is also possible to integrate a time-dependent infinitesimal
  generalized diffeomorphism. From $(X_t,b_t,a_t)$, we get
  $B_t=\int_0^t (f_s^* b_s + f_s^* a_s\wedge A_s) ds$ and
  $A_t=\int_0^t f_s^* a_s ds$, using a method analogous to that used to show Proposition 2.3
  in \cite{gualtieri-annals:2011}.
\end{remark}

\begin{remark}\label{rem:Dorfmann-product}
  We map $\cC^\infty(T\op T^*\op 1)$ to $\fgdiff(M)$ by
$$ X + \xi + \lambda \; \mapsto \; (X, d\xi, d\lambda),$$
so that we regard $X+\xi+\lambda$ as defining an infinitesimal
generalized diffeomorphism. Its natural action on sections of $T\op
T^*\op 1$ gives an action of a generalized vector field on generalized
vector fields, called the Dorfman product
$$ (X + \xi +\lambda) (Y + \eta +\mu) = [X,Y] + \cL_X \eta  + i_X d\mu - i_Y d\xi + 2\mu d\lambda - i_Y d\lambda. $$
The antisymmetrization of the Dorfman product gives the Courant
bracket defined above.
\end{remark}

\subsection{Differential forms as spinors}

By analogy with $D_n$-generalized geometry, the differential forms
$\bigwedge^\bullet T^*M$ are a Clifford module over the algebra
$\cC^\infty(Cl(T\op T^*\op 1))$ with an action defined~by
$$(X+\xi+\lambda)\cdot \phi = i_X \phi + \xi\wedge \phi + \lambda \tau \phi,$$
where $\tau \phi = \phi^{+} - \phi^{-}$ for the even $\phi^+$ and odd
$\phi^-$ parts of $\phi$. Thus, $\tau$ defines an involution of
$\bigwedge^\bullet T^*M$. The action defined above satisfies the
Clifford condition $(X+\xi+\lambda)^2\cdot \phi =
(X+\xi+\lambda,X+\xi+\lambda) \phi,$ as $\tau$ anticommutes with
interior and exterior products.

The action of $B$ and $A$ fields, $B,A\in \cC^\infty(\fso(T\op T^*\op
1))$, on $\bigwedge^\bullet T^*M$ via the spinorial representation
$\sigma:\cC^\infty(\Spin(T\op T^*\op 1))\to \Aut \textstyle
(\bigwedge^\bullet T^*M)$ is given by the Lie algebra action $
\sigma_*(B)\phi =-B\wedge \phi$, $\sigma_*(A)\phi =-A\wedge \tau
\phi,$ and the Lie group action
\begin{align*}
  \sigma(\exp B)\phi & = \phi - B\wedge \phi + B^2\wedge \phi + \ldots = e^{-B}\wedge \phi,\\
  \sigma(\exp A)\phi & = \phi -A\wedge \tau \phi = e^{-A\tau}\wedge
  \phi.
\end{align*}
Since $B$ and $A$ commute, the action of a $B+A$-field is given by
$$\sigma(\exp(B+A))\phi=e^{-B}e^{-A\tau}\phi=e^{-A\tau}e^{-B}\phi.$$

The Lie derivative of a spinor with respect to a generalized vector
field $X+\xi+\lambda$, as also for generalized vector fields in Remark
\ref{rem:Dorfmann-product}, is defined by mapping the vector field to
the infinitesimal generalized diffeomorphism $(X,d\xi,d\lambda)\in
\fgdiff(M)$ and differentiating the action of the one-parameter
subgroup $\{F_t\}$ to which it integrates:
$$\mathbf{L}_{X+\xi+\lambda} \phi = -\frac{d}{dt}_{|t=0} F_t \phi = \cL_X \phi + d\xi\wedge \phi + d\lambda \tau \phi.$$

The Lie derivative of a spinor satisfies a Cartan formula, where the
interior product is replaced by the Clifford action, $d
((X+\xi+\lambda)\cdot \phi) + (X+\xi+\lambda)\cdot
d\phi=\mathbf{L}_{X+\xi+\lambda}\phi$.

The differential forms $\bigwedge^\bullet T^*M$ are endowed with an
$\SO(T\op T^*\op 1)$-invariant pairing with values in $\bigwedge^n
T^*M$ coming from the Chevalley pairing on spinors
(\cite{chevalley-spinors:1954}). Let $\alpha$ be the anti-involution
defined by $\alpha(\omega)=(-1)^{{\deg \omega \choose 2}}\omega$ on
forms of pure degree $\omega$ and extended linearly. For $\rk T=\dim
M$ odd, the pairing is given by
\begin{equation*}
  \la \phi, \psi \ra =  [\al(\phi^-)\wedge \psi^+ - \al(\phi^+)\wedge \psi^- ]_{top}, 
\end{equation*}
while for $\rk T=\dim M$ even, it is given by
$$\la \phi,\psi \ra = [\alpha(\phi^+)\wedge \psi^+ + \alpha(\phi^-)\wedge \phi^-]_{top}.$$

\begin{remark}\label{rem:pairing-3-man}
  In the case of $3$-manifolds, \begin{align*}
    \la \phi,\psi \ra & =[\alpha(\phi^+)\wedge \psi^- - \alpha(\phi^-)\wedge \psi^+]_{top} \\ & =[(\phi_0 - \phi_2)\wedge (\psi_1+\psi_3) - (\phi_1-\phi_3)\wedge(\psi_0+\psi_2)]_{top} = \\
    & = \phi_0 \psi_3 + \psi_0 \phi_3 - \phi_1\wedge \psi_2 - \psi_1
    \wedge \phi_2,
  \end{align*}
  and, in particular, $\la \phi,\phi \ra = 2(\phi^0 \phi^3 -
  \phi^1\wedge \phi^2 )$, thus defining a quadratic form of signature
  $(4,4)$.
\end{remark}

\section{$G_2^2$-structures on $3$-manifolds}
\label{sec:G22-str}

In \cite{hitchin-CY:2003}, for $n=2m$, generalized Calabi-Yau
structures are defined by a complex closed form $\phi$ that is either even or
odd which is a pure spinor and satisfies $\la \phi,\bar{\phi}\ra \neq
0$. This structure defines a reduction to the stabilizer of the
spinor~field, $\SU(m,m)$.

In the case of a $3$-manifold, we pointwise have a seven-dimensional
generalized tangent space with an inner product of signature $(4,3)$. Its
space of spinors is eight-dimensional and equipped with a signature
$(4,4)$ inner product. In this setting, pure spinors correspond to
null spinors with respect to the inner product, while non-pure spinors
correspond to non-null spinors. Moreover, up to scalar multiplication, there are only
two orbits under the action of $\Spin(4,3)$: the null ones and the
non-null ones. Hence, all non-null spinors have isomorphic
stabilizers. While the stabilizer of a non-zero spinor in $\Spin(7)$
is the compact exceptional Lie group $G_2$, for the group
$\Spin(4,3)$, the stabilizer of a non-null spinor is its non-compact
real form $G_2^2$. The study of the structure given on a $3$-manifold
by a section of $\bigwedge^{\bullet} T^*M$ consisting of closed
non-null spinors motivates the following definition.

\begin{definition}
  A $G_2^2$-generalized structure on a $3$-manifold $M$ is an
  everywhere non-null section of the real spinor bundle, $\rho \in
  \Omega^\bullet(M)$, such that $d\rho = 0$.  For the sake of brevity,
  we call them $G_2^2$-structures.
\end{definition}

\begin{remark}
  Given a section $\rho \in \Omega^\bullet(M)$ consisting of closed
  null spinors, its annihilator $\Ann(\rho)\subset T\op T^*\op 1$
  defines an integrable real Dirac structure, i.e., a maximal
  isotropic subbundle of $T\op T^*\op 1$ involutive with respect to
  the Courant bracket. The involutivity is a consequence of the
  closedness of $\rho$, as in Proposition 1 of \cite{hitchin-CY:2003}.
\end{remark}

\subsection{Existence of $G_2^2$-structures}

From the non-nullity condition we have that $\la \rho,\rho \ra=
2(\rho_0\rho_3 - \rho_1\wedge\rho_2)$ defines a volume form on $M$, so
$G_2^2$-structures only exist over orientable manifolds. In fact,
given any volume form $\omega$, $c+\omega$ defines a $G_2^2$-structure
for any constant $c\neq 0$. Since $\rho$ is closed, $\rho_0$ must be a constant.

From now on, $M$ will denote a compact orientable $3$-manifold. Let $\GDiff^+(M)$ be the group of orientation-preserving generalized diffeomorphisms.

\begin{proposition} Up to $\GDiff^+(M)$-equivalence, a $G_2^2$-structure
  $\rho$ with $\rho_0\neq 0$ on $M$
  \begin{enumerate}
  \item is of the form $c+\omega$ for $c\neq 0$ and $\omega$ a volume
    form, and
  \item is completely determined by the cohomology
    classes $$([\rho_0],[\la \rho,\rho \ra])\in (H^0(M,\R)\setminus
    \{0\}) \op (H^3(M,\R)\setminus \{0\}).$$
  \end{enumerate}
\end{proposition}

\begin{proof}
  Let $\rho = \rho_0 + \rho_1 + \rho_2 + \rho_3$ be a
  $G_2^2$-structure with $\rho_0\neq 0$. It is equivalent, by the
  action of the closed $(B+A)$-field $\left(-\frac{\rho_1}{\rho_0},
    -\frac{\rho_2}{\rho_0}\right)$ to $$\rho_0 +
  \frac{1}{\rho_0}(\rho_0\rho_3 - \rho_1\wedge \rho_2)=\rho_0 +
  \frac{1}{2\rho_0}\la \rho,\rho \ra,$$ which is of the form
  $c+\omega$ for $c\neq 0$ and $\omega$ a volume form, as stated in
  the first part.  By Moser's theorem (\cite{moser-1965}), any two volume
  forms in the same cohomology class are diffeomorphic.\end{proof}

We deal now with the existence of $G_2^2$-structures with $\rho_0= 0$.
\begin{proposition}\label{prop:G22-str-rho0=0}
  If a compact $3$-manifold is endowed with a $G_2^2$-structure such that
  $\rho_0=0$, then it is diffeomorphic to the mapping torus of a
  symplectic surface by a symplectomorphism.  Conversely, any such
  mapping torus can be endowed with a $G_2^2$-structure with
  $\rho_0=0$.
\end{proposition}

\begin{proof}
  From $\rho_0 = 0$ and $\la \rho,\rho \ra \neq 0$ we get $\rho_1\wedge
  \rho_2\neq 0$, so we have nowhere vanishing closed $1$-forms and $2$-forms
  $\rho_1$ and $\rho_2$. We can perform a small deformation on
  $\rho_1$ to give it rational periods (as shown for instance in
  \cite{tischler:1970}). A suitable multiple has integral periods and
  defines a fibration $\pi:M\to \SSS^1$. To define $\pi$, take a base
  point $m\in M$ and let $\pi(x)=e^{2\pi i \int_{c(t)} \rho_1 dt}$
  where $c(t)$ is any curve joining $m$ and $x$. Let $X$ be the unique
  vector field satisfying $i_X\rho_2=0$ and $i_X \rho_1=1$ (so it
  is transversal to the fibration, $d\pi(X)\neq 0$).  Integrate the
  vector field $X$ to a one-parameter subgroup of diffeomorphisms
  $\{f_t\}$ such that $f_0=id$. Let $S$ be the fibre over the point
  $m\in M$. By the transversality, we have that $M$ is diffeomorphic
  to the mapping torus of $f_1$, i.e., the manifold $$\frac{S \ts
    [0,1]}{\{(x,0) \sim (f_1(x),1)\}_{x\in S}}.$$ The diffeomorphism
  is given by $[(y,t)]\mapsto f_t(y)\in M$.  Furthermore, $\cL_X
  \rho_2=d( i_X\rho_2)=0$, so $f_t^* \rho_2 = \rho_2$ and the fibres
  have a symplectic structure given by the restriction of $\rho_2$,
  which is closed and non-degenerate in every fibre $f_t(S)$. Thus,
  $S$ is a symplectic manifold and $f_1$ is a symplectomorphism.

  For the second part, let $M_f$ be the mapping torus of an orientable
  surface $(S,\omega)$ by a symplectomorphism $f$. We define a
  $2$-form $\rho_2$ on $M_f$ as the form which is fibrewise
  $\omega$. The form $\rho_2$ is well defined since
  $f^*\omega=\omega$. Let $\rho_1$ be the pullback of a non-vanishing
  $1$-form over the circle. The form $\rho_1 + \rho_2$ then defines a
  $G_2^2$-structure on $M_f$.
\end{proof}

\begin{lemma}\label{lemma:map-torus-difeo}
  The mapping torus of an orientable surface $S$ by an orientation-preserving diffeomorphism is diffeomorphic to the mapping torus of
  $S$ by a symplectomorphism.
\end{lemma}

\begin{proof}
  Let $f$ be the orientation-preserving diffeomorphism and let
  $\omega$ be a volume form of the surface $S$. The $2$-forms
  $f^*\omega$ and $\omega$ have the same volume and hence define the
  same cohomology class in $H^2(S,\R)$. We apply Moser's argument
  (\cite{moser-1965}) to the family $\omega_t=t\omega + (1-t)
  f^*\omega$, so we get a family of diffeomorphisms $\{\phi_t\}$, with
  $\phi_0=\id$, such that $\phi_t^* \omega_t=\omega$. Then, we have
  that $(\phi_1\circ f)^*=\phi_1^* f^*\omega=\omega$, i.e.,
  $\phi_1\circ f$ is a symplectomorphism, and $\{\phi_t\circ f\}$
  defines a diffeotopy between $f$ and $\phi_1\circ f$ which makes the
  mapping torus of $f$ diffeomorphic to the mapping torus of
  $\phi_1\circ f$.
\end{proof}

The following theorem is a consequence of the two previous results.

\begin{theorem}\label{theo:map-tor-twisted-G22-str}
  A compact $3$-manifold $M$ admits a $G_2^2$-structure with $\rho_0=0$ if and
  only if $M$ is the mapping torus of an orientable surface by an
  orientation-preserving diffeomorphism.
\end{theorem}

\begin{remark}
  From a $G_2^2$-structure with $\rho_0= 0$ on a $3$-manifold $M$ we
  define a symplectic structure on $M\times \SSS^1$ by $\rho_2 +
  \rho_1 \wedge d\theta$, where $d\theta$ denotes the usual $1$-form
  on $\SSS^1$ and we really mean the pullbacks of forms on $M$ and
  $\SSS^1$ to $M\times \SSS^1$. More generally, the condition that a
  $3$-manifold $M$ fibres over the circle is equivalent to the
  existence of a symplectic structure on $M\times S^1$, as addressed
  in \cite{friedl-vidussi:2011}.
\end{remark}

\begin{remark}
  After acting by a generalized diffeomorphism, a $G_2^2$-structure
  $\rho$ with $\rho_0=0$ can be written as $\rho_1 + \rho_2$. This is
  a co-symplectic structure on the $3$-manifold in the sense of
  \cite{libermann:1959}. In this context, statements similar to the
  ones in this section have been obtained in \cite{li-hongjun:2008}.
\end{remark}

\subsection{Deformation of $G_2^2$-structures}
\label{subsec:deformation-G22}

Inspired by the Moser argument for symplectic geometry, we study whether a
small perturbation of a $G_2^2$-structure (on a compact $3$-manifold $M$) within its cohomology class
may change the $G_2^2$-structure up to equivalence by
$$\GDiff_0(M) = \{ f\ltimes (B,A) \in \GDiff(M) \st f\in \Diff_0(M), B \textrm{ and } A \textrm{ are exact} \}.$$

Let $\rho^0, \rho^1\in \Omega^\bullet(M)$ be two
$G_2^2$-structures representing the same cohomology class,
$\rho^1-\rho^0=d\phi$, and sufficiently close to have that each form
$\rho^t=\rho^0 + t (\rho^1-\rho^0)$ is a $G_2^2$-structure, i.e., $\la
\rho^t,\rho^t \ra \neq 0$, for $0\leq t\leq 1$. We would like to have
a one-parameter family of generalized diffeomorphisms $\{ F_t \}$ such
that $F_t^* \rho^t=\rho^0$, making equivalent all the
$G_2^2$-structures between $\rho^0$ and $\rho^1$. We will be looking
for $\{F_t\}$ coming from a time-dependent generalized vector field
$\{X_t+\xi_t+\lambda_t\}$. By differentiating $F_t^* \rho^t=\rho^0$
and using Cartan's formula, we then have
$$ 0 = \frac{d}{dt} [F_t^* \rho^t] = F_t^* \left[ \frac{d\rho_t}{dt} + \mathbf{L}_{X_t+ \xi_t + \lambda_t} \rho^t \right] = F_t^*[ d\phi + d((X_t+\xi_t+\lambda_t)\cdot \rho^t) ]=0.$$
So, in order to find such generalized vector fields it will suffice to
solve the equation $d ((X_t+\xi_t+\lambda_t)\cdot \rho^t)) =
d(-\phi)$, or equivalently, to solve the equation
$(X_t+\xi_t+\lambda_t)\cdot \rho^t = -\phi$ where we are allowed to
modify $\phi$ by the addition of a closed form depending on $t$. This
latter equation corresponds to $\phi$ being in the image of the
Clifford product of the sections of the rank $7$ vector bundle $T\op
T^*\op 1$ by $\rho^t$. The spinor $\rho^t$ defines a map $T\op T^*\op
1 \to \bigwedge^\bullet T^*M$. Since $\rho^t$ is non-null, this map is
injective (the annihilator of a non-null spinor is trivial). From the
antisymmetry of the Clifford product with respect to the pairing, $\la
v_m\cdot \rho^t_m, \rho^t_m \ra_m = 0$, where $v_m$ and $\psi_m$ lie
over $m\in M$, and the image is $\{\rho^t\}^\perp=\{ \psi \in
\bigwedge^\bullet T^*M \st \la \rho^t,\psi \ra = 0 \}$. Thus, $\rho^t$
defines an isomorphism between the rank $7$ vector bundles $T\op
T^*\op 1$ and $\{\rho^t\}^\perp$. Consequently, for the equation
$(X_t+\xi_t+\lambda_t)\cdot \rho^t=-\phi$ to have a solution and then
apply the Moser argument, we must have $\phi\in
\cC^\infty(\{\rho^t\}^\perp)$.

\begin{proposition}\label{prop:Moser-G22-neq0}
  Any sufficiently small perturbation $\{\rho^t\}$ within the
  cohomology class of a $G_2^2$-structure $\rho^0$ such that
  $\rho^0_0\neq 0$ is equivalent to $\rho^0$ under the action of the group
  $\GDiff_0(M)$.
\end{proposition}

\begin{proof}
  We have that $\rho^t_0=\rho^0_0\neq 0$. Since we can add any closed
  form to $\phi$, we can arbitrarily modify its degree $3$ part. The
  Moser argument applies by setting $\phi^t_3 = - \frac{1}{\rho^0_0}
  \la \rho^t, \phi_o + \phi_1 + \phi_2 \ra$, so that we have $\la \rho^t,
  \phi^t \ra = 0$.
\end{proof}

When $\rho_0 =0$, the result remains true but involves some
technicalities.

\begin{lemma}\label{lemma:Moser-G22}
  Let $\rho$ be a $G_2^2$-structure with $\rho_0=0$ and $[\rho_1]\in
  H^1(M,\Q)$. There exists an operator $R:\Omega^\bullet(M)\to
  \Omega^\bullet_{cl}(M)$ such that $\phi+R\phi\in
  \cC^\infty(\{\rho\}^\perp)$.
\end{lemma}

\begin{proof}
  By considering a multiple of $\rho$ we can consider $[\rho_1] \in
  H^1(M,\Z)$. By Proposition \ref{prop:G22-str-rho0=0}, $M$ fibres
  over the circle with fibre $S$. First, define the constant $c=[\la
  \rho,\phi\ra] / [\rho_1\wedge \rho_2]$. Add the closed form
  $c\rho_2$ to $\phi$; then the cohomology class of $\la \rho,
  \phi+c\rho_2\ra$ is trivial. Thus, $\la \rho, \phi+c\rho_2\ra =
  d\alpha$ for some $2$-form $\alpha$. Choose a metric on $M$. Using
  the Hodge decomposition, the codifferential $d^*$ and the Green
  operator $G$, we may take $\alpha=d^*G\la \rho,\phi'\ra$. Integrate
  $\alpha$ over the fibres to get a function $g$ on the circle. Since
  $\rho_1\wedge\rho_2\neq 0$, the fibres are homologous and $\rho_2$
  is closed, then $\int_S \rho_2=c'\neq 0$ for any fibre $S$. Let
  $f=g/c'$. The $2$-form $\alpha_0=\alpha - f \rho_2$ has zero
  integral along the fibres. The metric on $M$ induces a metric on any
  fibre $S$, for which we define the codifferential $d^*_{S}$,
  harmonic operator $H_{S}$ and Green operator $G_{S}$ such
  that $$\alpha_{0|S}=H_{S}\alpha_{0|S}+d_{S}(d_{S}^* G_{S}
  \alpha_{0|S}) + d^*_{S}(d_{S} G_{S} \alpha_{0|S}).$$ For degree
  reasons, $d_{S} G_{S} \alpha_{0|S}=0$, and from $\int_{S}
  \alpha_{0|S}=0$, $H_{S}\alpha_{0|S}=0$. We then have, over each
  fibre $S$, $\alpha_{0|S}=d_{S}\beta$ where $\beta=d_{S}^* G_{S}
  \alpha_{0|S}$. Since the metric on $M$ determines a smoothly varying
  family of metrics over the fibres, we have a globally smooth
  $1$-form $\beta$ such that $\alpha_0-d\beta$ is zero restricted to a
  fibre.

  Let $X$ be the vector field transversal to the fibration such that
  $i_X \rho_1=1$, and let $\gamma= - i_X(\alpha_0 - d\beta)$. We have
  that $\alpha_0 - d\beta = \gamma \wedge \rho_1$. By differentiating
  this expression we get $$d\alpha = d(\alpha_0 + f\rho_2) = df \wedge
  \rho_2 + \rho_1 \wedge d\gamma.$$

  Define ${\rm R}\phi=c\rho_2 + df + d\gamma\in
  \Omega^\bullet_{cl}(M)$. Since $c$, $f$ and $\gamma$ have been
  uniquely defined, ${\rm R}$ defines an operator on differential
  forms. We have by construction that $\la \rho, \phi + {\rm R}\phi
  \ra = 0$, i.e., $\phi + {\rm R}\phi\in
  \cC^\infty(\{\rho\}^\perp)$.\end{proof}

Let ${\rm Q}\phi \in \cC^\infty(T\op T^*\op 1)$ be the unique
generalized vector field such that ${\rm Q}\phi \cdot \rho = - (\phi +
{\rm R}\phi)$. Thus ${\rm Q}$ defines an operator
$\Omega^\bullet(M)\to\cC^\infty(T\op T^*\op 1)$.

  \begin{proposition}\label{prop:Moser-G22-0}
    Any sufficiently small perturbation $\{\rho^t\}$ within the
    cohomology class of a $G_2^2$-structure $\rho^0$ such that
    $\rho^0_0=0$ is equivalent to $\rho^0$ by $\GDiff_0(M)$.
  \end{proposition}

  \begin{proof}

    When $[\rho^0_1]\in H^1(M,\Q)$, we use Lemma \ref{lemma:Moser-G22}
    to produce an operator ${\rm R}_t$ for each $\rho^t$ and we define
    $\phi^t=\phi+{\rm R}_t \phi$, so that $\la \rho^t, \phi^t\ra=0$
    and the Moser argument applies.

    For the general case, we prove an analogous result in a
    neighbourhood of a $G_2^2$-structure with rational degree $1$ part
    and use a density argument. We drop the superindex $t$ for the
    sake of brevity. Consider $\rho + \lambda \beta$, with $\lambda>0$
    and $\beta\in\Omega^\bullet_{cl}(M)$ such that $\beta_0=0$. We
    want to solve the equation $u\cdot (\rho+\lambda \beta) = -\phi$
    up to addition of closed forms. To do that, consider
    \begin{equation}\label{eq1}
      (u_0+\lambda u_1 + \lambda^2 u_2 + \ldots )\cdot (\rho+\lambda\beta) = -(\phi + {\rm R}\phi + \lambda \gamma_1 + \lambda^2 \gamma_2 + \ldots),\tag{$\star$}
    \end{equation}
    for closed forms $\gamma_i$. We solve it iteratively, starting
    with $u_0\cdot \rho = -\phi + {\rm R}\phi$, which has solution
    $u_0={\rm Q}\phi$. We then have $u_1\cdot \rho = -({\rm
      Q}\phi\cdot \beta + \gamma_1)$. We define the operator ${\rm
      P}:\Omega^\bullet(M)\to \Omega^\bullet(M)$ by ${\rm P}\phi={\rm
      Q}\phi\cdot \beta$ and consider $\gamma_1 = {\rm RP}\phi$. The
    equation becomes $u_1\cdot \rho = -({\rm P}\phi+{\rm RP}\phi)$,
    whose solution is $u_1={\rm QP}\phi$. For $j\geq 2$ we have
    $u_j\cdot \rho = - u_{j-1} \cdot \beta + \gamma_j = - {\rm
      P}^{j}\phi + \gamma_j$. By taking $\gamma_j={\rm RP}^j\phi$, the
    solution is given by $u_j={\rm QP}^j\phi$. We thus obtain a formal
    solution of \eqref{eq1} by
$${\rm Q}(\phi+\lambda {\rm P}\phi+\lambda^2{\rm P}^2\phi+\ldots)\cdot (\rho+\lambda\beta)=-\phi+{\rm R}(\phi+\lambda {\rm P}\phi +\lambda^2 {\rm P}^2\phi +\ldots).$$
To see the convergence of the series $\phi+\sum_{j=1}^\infty \lambda^j
{\rm P}^j\phi$ for $\lambda$ sufficiently small, we consider Sobolev
spaces ${\rm H}_s(T\op T^*\op 1)$ and ${\rm
  H}_s(\bigwedge^\bullet(M))$ with norms $||\;||_s$. Since the
operator ${\rm Q}$ is defined in terms of the Green operator and
integration over the fibres, it is bounded, and so is the operator
${\rm P}$. For $s$ sufficiently large and any $\beta$ such that
$||v\cdot \beta||_s\leq ||v||_s$, there exists some constant $C_s$
such that $||{\rm P}\phi||_s\leq C_s ||\phi||_s$.

Take $\lambda$ such that $0<\lambda<\frac{1}{2C_s}$. Then,
$\phi+\sum_{j=1}^\infty \lambda^j {\rm P}^j\phi$ is a Cauchy sequence
and converges to a form $\Phi\in {\rm
  H}_s(\bigwedge^\bullet(M))$. Equation \eqref{eq1} becomes $u\cdot
(\rho+\lambda\beta)=-(\phi+{\rm R}\Phi)$ and a solution is given by
${\rm Q}\Phi \in {\rm H}_s(T\op T^*\op 1)$.

We have that for any $\rho$ such that $[\rho_1]\in H^1(M,\Q)$, there
exists a neighbourhood for which there is a solution in ${\rm H}_s(T\op
T^*\op 1)$. Since $\phi\in\Omega^\bullet(M)$ belongs to ${\rm
  H}_s(\bigwedge^\bullet(M))$ for any $s$, we have that the solution
belongs to ${\rm H}_s$ for any $s$. Thus, the series defines $\Phi\in
\cC^\infty(\bigwedge^\bullet(M))$, we have that ${\rm Q}\Phi\in \cC^\infty(T\op
T^*\op 1)$ is a solution of $u\cdot \rho^t = -\phi$ up to closed
forms, and the Moser argument applies. Since there exists a solution
in an open neighbourhood of any rational form, by density of the
rational forms, there exists a solution for any closed form $\rho$ and
the Moser argument applies.
\end{proof}

We summarize Propositions \ref{prop:Moser-G22-neq0} and
\ref{prop:Moser-G22-0} in the following theorem.

\begin{theorem}\label{theo:small-perturbation}
  Any sufficiently small perturbation $\{\rho^t\}$ within the
  cohomology class of a $G_2^2$-structure $\rho^0$ is equivalent to
  $\rho^0$ by $\GDiff_0(M)$.
\end{theorem}

\newpage

\subsection{The cone of $G_2^2$-structures}
\label{subsec:cone-G22}

Inspired by the cones of K\"ahler and symplectic structures inside the
\mbox{second} cohomology group of a manifold, we raise a similar
question for $G_2^2$-structures on compact $3$-manifolds. What are the cohomology classes
$[\rho]\in H^\bullet(M,\R)$ which have a representative in
$\Omega^\bullet(M,\R)$ defining a $G_2^2$-structure compatible with
the orientation of $M$? From the homogeneity of the condition $\la
\rho, \rho \ra> 0$, it is clear that these elements form an open cone
in $H^\bullet(M,\R)$.

Consider a mixed degree cohomology class $[\rho]\in H^\bullet(M,\R)$
satisfying $[\rho_o][\rho_3]-[\rho_1][\rho_2]> 0 \in
H^\bullet(M,\R)$. In the case that $[\rho_0]\neq 0$, i.e., $\rho_0\neq
0$, consider a non-vanishing form $\omega$ representing the degree $3$
class $[\rho_o \rho_3 - \rho_1 \wedge \rho_2]$. Define
$\rho'=\rho_0+\rho_1+\rho_2 + \frac{1}{\rho_0}(\omega + \rho_1 \wedge
\rho_2)$, which satisfies $\la \rho',\rho'\ra=2\omega$ and is thus a
$G_2^2$-structure representing $[\rho]$.

On the other hand, for a class $[\rho]$ with $[\rho_0]=0$, i.e.,
$\rho_0= 0$, the condition $[\la \rho,\rho\ra]=-2[\rho_1][\rho_2]> 0$
must be satisfied. Moreover, $[\rho_1]$ and $[\rho_2]$ must be
represented by non-vanishing forms.  From Theorem 5 in
\cite{thurston-memoirs:1986}, the set of cohomology classes $C_1$ in
$H^1(M,\R)$ which can be represented by a non-singular closed $1$-form
constitutes an open set described as follows. Define the norm $X$ for $\omega
\in H^2(M,\R)$ as the infimum of the negative parts of the Euler
characteristics of embedded surfaces defining $\omega$, and extend
this definition to $H^1(M,\R)$ using Poincar\'e duality. Namely, the
norm of a $1$-form $\phi$ in $M$
is $$||\phi||_X=\textrm{min}\{\chi_-(S)\st S\subset M \textrm{
  properly embedded surface dual to } \phi \},$$ where
$\chi_-(S)=\textrm{max}\{-\chi(S),0\}$. The unit ball for this norm is
a polytope called the Thurston ball $B_X$. The set of $1$-cohomology
classes $C_1$ represented by non-vanishing $1$-forms consists of the
union of the cones on some open faces, so-called fibred faces, of the
Thurston ball, minus the origin.

For each element $\al=[a]\in C_1$, given by a non-singular $a$, take
$h\in H^2(M,\R)$ such that $h\cup \al>0$. Lemma 2.2 in
\cite{friedl-vidussi:2012} ensures that we can always find a
representative $\Omega$ of the class $h$, such that $\Omega\wedge
a>0$. Hence, if we define $$C=\{ (\al,\be)\in C_1\op H^2(M,R) \st
\al\cup \be< 0 \},$$ we have that the cone of $G_2^2$-structures with
$\rho_0=0$ in $H^\bullet(M,\R)$ is given by $C \op H^3(M,\R).$ To sum
up, we have the following theorem.

\begin{theorem}\label{theo:G22-cone}
  The cone of $G_2^2$-structures, or $G_2^2$-cone, is given by
$$ \{ [\rho]\in H^\bullet(M,\R) \st [\rho_0]\neq 0 \textrm{ and } [\rho_0][\rho_3] - [\rho_1][\rho_2]>0 \} \bigcup \left(C\op H^3(M,\R) \right).$$
\end{theorem}

\newpage

\bibliographystyle{alpha}
\bibliography{GenGeo}

\end{document}